\documentclass{amsart}[12pt]
\usepackage[utf8]{inputenc}
\usepackage{amsmath}
\usepackage{amsfonts}
\usepackage{amssymb}

\newtheorem{theorem}{Theorem}
\newtheorem{lemma}{Lemma}

\theoremstyle{definition}

\newtheorem{remark}{Remark}

\begin{document}
\title[Strict singularity of weighted composition operators]%
{Strict singularity of weighted composition operators on derivative Hardy spaces}

\author{Qingze Lin, Junming Liu*, Yutian Wu}
\thanks{*Corresponding author}

\address{School of Applied Mathematics, Guangdong University of Technology, Guangzhou, Guangdong, 510520, P.~R.~China}\email{gdlqz@e.gzhu.edu.cn}

\address{School of Applied Mathematics, Guangdong University of Technology, Guangzhou, Guangdong, 510520, P.~R.~China}\email{jmliu@gdut.edu.cn}

\address{School of Financial Mathematics \& Statistics, Guangdong University of Finance, Guangzhou, Guangdong, 510521, P.~R.~China}\email{26-080@gduf.edu.cn}

\begin{abstract}
We prove that the weighted composition operator $W_{\phi,\varphi}$ fixes an isomorphic copy of $\ell^p$ if the operator $W_{\phi,\varphi}$ is not compact on the derivative Hardy space $S^p$. In particular, this implies that the strict singularity of the operator $W_{\phi,\varphi}$ coincides with the compactness of it on $S^p$. Moreover, when $p\neq2$, we characterize the conditions for those weighted composition operators $W_{\phi,\varphi}$ on $S^p$ which fix an isomorphic copy of $\ell^2$ .
\end{abstract}
\thanks{This work was supported by NNSF of China (Grant No. 11801094).}

\maketitle

\section{\bf Introduction}
Let $\mathbb{D}$ denote the open unit disk in the complex plane $\mathbb{C}$,  and $H(\mathbb D)$ the space of all analytic functions in $\mathbb{D}$. For $0<p<\infty$, the Hardy space $H^{p}$ is the space of functions $f\in H(\mathbb D)$ for which
$$ \|f\|_{H^{p}}:=\left(\sup_{0\leq r<1}\int_{\partial\mathbb{D}}|f(r\xi)|^{p}dm(\xi)\right)^{1/p}<\infty\,,$$
where $m$ is the normalized Lebesgue measure on $\partial\mathbb{D}$. From \cite[Theorem~9.4]{zhu}, this norm is equal to the following norm:
$$\|f\|_{H^{p}}=\left(\int_{\partial\mathbb{D}}|f(\xi)|^{p}dm(\xi)\right)^{1/p}\,,$$
where for any $\xi\in\partial\mathbb{D}$, $f(\xi)$ is the radial limit which exists almost every.

For $p=\infty$, the space $H^{\infty}$ is defined by
$$H^{\infty}=\{f\in H(\mathbb D):\ \|f\|_{\infty}:=\ \sup_{z\in \mathbb D}\{|f(z)|\}<\infty\}\,.$$

We define the weighted composition operator $W_{\phi,\varphi}$ for $f\in H(\mathbb D)$ by
$$W_{\phi,\varphi}(f)(z)=\phi(z)f\circ\varphi(z)\,,\qquad z\in \mathbb D\,,$$
where $\phi\in H(\mathbb D)$ and $\varphi$ is an analytic self-map of $\mathbb D$.
If $\phi(z)\equiv 1$, $W_{\phi,\varphi}$ becomes the composition operator $C_\varphi$ while
if $\varphi(z)\equiv z$, $W_{\phi,\varphi}$ becomes the multiplication operator $M_\phi$.
For weighted composition operators $W_{\phi,\varphi}$ on Hardy spaces $H^p$, we refer the readers to the literatures \cite{CH1,CM,ZCRZ,RD}.

We define the derivative Hardy space $S^p$ by
$$S^p=\{f\in H(\mathbb{D}): \|f\|_{S^{p}}:=|f(0)|+\|f'\|_{H^{p}}<\infty\}.$$
For $1\leq p\leq\infty$, $S^p$ is a Banach algebra and there is an inclusion relation: $S^p\subset H^\infty$ (for the detail structures of $S^p$ spaces, see \cite{ZCBP,ZCBP1,GL,LIN,LIN1} for references).

In paper~\cite{RR}, Roan started the investigation of composition operators on the space $S^p$. After his work, MacCluer \cite{BDM} gave the characterizations of the boundedness and the compactness of the composition operators on the space $S^p$ in terms of Carleson measures. A remarkable result on the boundedness and the compactness of the weighted composition operators on $S^p$ was obtained in \cite{CH}, in which they are both characterized through the corresponding weighted composition operators on $H^p$. Furthermore, the isometry between $S^p$ was obtained by Novinger and Oberlin in \cite{NO}, in which they showed that the isometries were closely related to the weighted composition operator.

A bounded operator $T\colon X \to Y$ between Banach spaces is strictly singular if its restriction to any infinite-dimensional closed subspace is not an isomorphism onto its image. This notion was introduced by Kato \cite{Kato}.

A bounded operator $T\colon X \to Y$ between Banach spaces is said to fix a copy of the given Banach space $E$ if there is a closed subspace $M\subset X$, linearly isomorphic to $E$, such that the restriction $T_{|M}$ defines an isomorphism from $M$ onto $T(M)$. The bounded operator $T\colon X \to Y$ is called $\ell^p$-singular if it does not fix any copy of $\ell^p$\,.

Laitila, et al \cite{LNST} recently investigated the strict singularity for the composition operators on $H^p$ spaces. Following their ideas, Miihkinen \cite{SM} studied the strict singularity of $T_g$ on Hardy space $H^p$ and showed that the strict singularity of $T_g$ coincides with its compactness on $H^p, \, 1 \leq p < \infty.$ Miihkinen \cite{SM} also post an open question which was resolved in \cite{MNST} by utilizing the generalized Volterra operators. It should be noticed that Hern\'{a}ndez, et al \cite{HST} investigated the interpolation and extrapolation of strictly singular operators between $L_p$ spaces.

In this paper, we prove that the weighted composition operator $W_{\phi,\varphi}$ fixes an isomorphic copy of $\ell^p$ if the operator $W_{\phi,\varphi}$ is not compact on the derivative Hardy space $S^p$. In particular, this implies that the strict singularity of the operator $W_{\phi,\varphi}$ coincides with the compactness of it on $S^p$. Moreover, when $p\neq2$, we characterize the conditions for those weighted composition operators $W_{\phi,\varphi}$ on $S^p$ which fix an isomorphic copy of $\ell^2$ .

Our main results read as follows:

\begin{theorem}
\label{th1}
Let $1 \leq p <\infty$, $\phi\in H(\mathbb D)$ and $\varphi$ is an analytic self-map of $\mathbb D$. If the weighted composition operator $W_{\phi,\varphi}\colon S^p \to S^p$ is bounded but not compact, then $W_{\phi,\varphi}$ fixes an isomorphic copy of $\ell^p$ in $S^p$. In particular, the operator $W_{\phi,\varphi}$ is not strictly singular, that is, strict singularity of bounded operator $W_{\phi,\varphi}$ coincides with its compactness.
\end{theorem}
\begin{remark}
In the final section, we prove that the claims of theorem~\ref{th1} is still true for the case of $p=\infty$\,.
\end{remark}

Denote $E_\varphi = \{\zeta \in\partial\mathbb{D} : |\varphi(\zeta)|=1\}$, then we have
\begin{theorem}
\label{th2}
Let $1 \leq p <\infty$, $\phi\in H(\mathbb D)$ and $\varphi$ is an analytic self-map of $\mathbb D$.  Suppose that $W_{\phi,\varphi}\colon S^p \to S^p$ is bounded and $m(E_\varphi)=0$. If $W_{\phi,\varphi}$ is bounded below on an infinite-dimensional subspace $M\subset S^p$, then the restriction $W_{\phi,\varphi}$ on $M$ fixes an isomorphic copy of $\ell^p$ in $M$. In particular, if $p\neq2$, the operator $W_{\phi,\varphi}$ does not fix any isomorphic copy of $\ell^2$ in $S^p$.
\end{theorem}

When $m(E_\varphi)>0$, it holds that
\begin{theorem}
\label{th3}
Let $1 \leq p <\infty$, $\phi\in H(\mathbb D)$ and $\varphi$ is an analytic self-map of $\mathbb D$.  Suppose that $W_{\phi,\varphi}\colon S^p \to S^p$ is bounded. if $m(E_\varphi)>0$ and $\phi\varphi'\neq0$, then the operator $W_{\phi,\varphi}$ fixes an isomorphic copy of $\ell^2$ in $S^p$.
\end{theorem}

Notation: throughout this paper, $C$ will represents a positive constant which may be given different values at different occurrences.

\section{Proof of Theorem~\ref{th1}}
This section is devoted to the proof of Theorem~\ref{th1}. First, the following Lemma~\ref{le1} can be deduced from \cite[Theorem~2.1]{CH} and \cite[Theorem~2.2 and Theorem~2.3]{RD}.

\begin{lemma}
\label{le1}
Let $1 \leq p <\infty$, $\phi\in H(\mathbb{D})$ and $\varphi$ is an analytic self-map of $\mathbb D$. Then $W_{\phi,\varphi}\colon S^p \to S^p$ is compact if and only if $\phi\in S^p$ and
$$\lim_{|a|\rightarrow1^-}\int_{\partial\mathbb{D}}\frac{1-|a|^2}{|1-\bar{a}\varphi(\omega)|^2}|\psi(\omega)\varphi'(\omega)|^pdm(\omega)=0\,.$$
\end{lemma}

The following lemma~\ref{le2} is proven in \cite[Proposition~3.3(ii)]{CH}.

\begin{lemma}
\label{le2}
Let $1 \leq p \leq\infty$, $\phi\in H^p$ and $\varphi$ is an analytic self-map of $\mathbb D$. Then $W_{\phi,\varphi}\colon S^p \to H^p$ is compact.
\end{lemma}

We employ
the test functions
$$f_a(z) = \int_0^z\frac{(1-|a|^2)^{1/p}}{(1-\bar{a}\omega)^{2/p}}d\omega\,,
   \quad z \in \mathbb{D},$$
where $a \in \mathbb{D}$. They all satisfy $\|f_a\|_{S^p} = 1$ and $f_{a}$ converges to $0$ uniformly on compact subsets of $\mathbb{D}$, as $|a|\to1^-$\,.

Let $L=\{\xi\in\partial\mathbb{D}:\ \text{the radial limit }\varphi(\xi)\ \text{exists}\}$ and
$$E_{\varepsilon}:=\{\xi\in L:\ |1-\varphi(\xi)|<\varepsilon\}$$
for any given $\varepsilon>0$, then $m(\partial\mathbb{D}\setminus L)=0$. The proof of Theorem~\ref{th1} relies on the following auxiliary lemma.

\begin{lemma}
\label{le3}
Let $(a_n) \subset \mathbb{D}$ be a sequence such that $0 < |a_1| < |a_2| < \ldots < 1$ and $a_n \to 1$. If the bounded operator $W_{\phi,\varphi}\colon S^p \to S^p$ is not compact, then we have
\begin{align*}
&\textrm{(1) $\lim_{\varepsilon \to 0}\int_{E_\varepsilon}|(W_{\phi,\varphi} f_{a_n})'|^p dm = 0$\quad for every $n\in\mathbb{N}$.}&\\
&\textrm{(2) $\lim_{n \rightarrow \infty}\int_{\partial\mathbb{D}\setminus E_\varepsilon}|(W_{\phi,\varphi} f_{a_n})'|^p dm = 0$\quad for every $\varepsilon > 0.$}&
\end{align*}
\end{lemma}
\begin{proof}
\textbf{(1)} For each fixed $n$, this follows immediately from the absolute continuity of Lebesgue measure, the boundedness of operator $W_{\phi,\varphi}$ and the fact that $W_{\phi,\varphi}$ is not compact (which implies that $\varphi$ is not identically $1$)\,.

\textbf{(2)} For any given $\varepsilon > 0$, let $\xi\in L\setminus E_\varepsilon$. Then there exists an $N>0$ such that whenever $n>N$, it holds that
\begin{equation}\begin{split}\nonumber
|1-\bar{a}_n\varphi(\xi)|&=|1-\varphi(\xi)+\varphi(\xi)-\bar{a}_n\varphi(\xi)|\\
&\geq |1-\varphi(\xi)|-|\varphi(\xi)-\bar{a}_n\varphi(\xi)|\\
&\geq |1-\varphi(\xi)|-|1-\bar{a}_n|>\frac{\varepsilon}{2}\,.
\end{split}\end{equation}

Now, by definition, we have
$$\int_{\partial\mathbb{D}\setminus E_\varepsilon}|(W_{\phi,\varphi} f_{a_n})'|^pdm\leq C\left(\int_{\partial\mathbb{D}\setminus E_\varepsilon}|\phi'f_{a_n}(\varphi)|^p dm+\int_{\partial\mathbb{D}\setminus E_\varepsilon}|\phi\varphi'f'_{a_n}(\varphi)|^pdm\right).$$
Since $W_{\phi,\varphi}$ is bounded, it follows that $\phi\in S^p$, that is, $\phi'\in H^p$. By Lemma~\ref{le2}, $W_{\phi',\varphi}:\ S^p\to H^p$ is compact, which implies that
$$\lim_{n\to\infty}\int_{\partial\mathbb{D}\setminus E_\varepsilon}|\phi'f_{a_n}(\varphi)|^p dm\leq\lim_{n\to\infty}\int_{\partial\mathbb{D}}|W_{\phi',\varphi} f_{a_n}|^pdm=0\,.$$

For the estimate of the second integral, we have
\begin{equation}\begin{split}\nonumber
\int_{\partial\mathbb{D}\setminus E_\varepsilon}|\phi\varphi'f'_{a_n}(\varphi)|^pdm&=\int_{\partial\mathbb{D}\setminus E_\varepsilon}|\phi\varphi'|^p\frac{1-|a_n|^2}{|1-\bar{a}_n\varphi|^2}dm\\
&=\int_{\partial\mathbb{D}\setminus E_\varepsilon}|\phi\varphi'|^p\frac{1-|a_n|^2}{|1-\bar{a}_n\varphi|^2}dm\\
&\leq\frac{4(1-|a_n|^2)}{\varepsilon^2}\int_{\partial\mathbb{D}}|\phi\varphi'|^p dm\,,
\end{split}\end{equation}
where $\int_{\partial\mathbb{D}}|\phi\varphi'|^p dm$ is finite due to the boundedness of $W_{\phi,\varphi}:S^p\to S^p$ and \cite[Theorem~2.1]{CH} and \cite[Theorem~4]{ZCRZ}.

Therefore,
$$\lim_{n\to\infty}\int_{\partial\mathbb{D}\setminus E_\varepsilon}|\phi\varphi'f'_{a_n}(\varphi)|^pdm=0\,,$$
The proof is complete.
\end{proof}

Now, we are ready to give a proof of Theorem~\ref{th1}.
\begin{proof}[Proof of Theorem~\ref{th1}]
First, we prove that there exists a sequence $(a_n) \subset \mathbb{D}$ with $0 < |a_1| < |a_2| < \ldots < 1$ and $a_n \to \omega \in \partial\mathbb{D},$ such that there is a positive constant $h$ such that
$$\|W_{\phi,\varphi}(f_{a_n})\|_{H^p}\geq h>0$$
holds for all $n\in\mathbb{N}$\,.

Since $W_{\phi,\varphi}:S^p\to S^p$ is not compact, Lemma~\ref{le1}, there exists a sequence $(a_n) \subset \mathbb{D}$ with $0 < |a_1| < |a_2| < \ldots < 1$ and $a_n \to \omega \in \partial\mathbb{D},$ such that there is a positive constant $h$ such that
$\|\phi\varphi'f'_{a_n}(\varphi)\|_{H^p}\geq 2h>0$ holds for all $n\in\mathbb{N}$\,. Note that
$$\|W_{\phi,\varphi}(f_{a_n})\|_{H^p}\geq\|\phi\varphi'f'_{a_n}(\varphi)\|_{H^p}-\|\phi'f_{a_n}(\varphi)\|_{H^p}\,.$$
By Lemma~\ref{le2}, $W_{\phi',\varphi}:\ S^p\to H^p$ is compact, which implies that
$$\lim_{n\to\infty}\|\phi'f_{a_n}(\varphi)\|_{H^p}=0\,.$$
Hence, there exists a subsequence of $(a_n)$ (denoted still by $(a_n)$) such that the above claim holds. We assume without loss of generality that $a_n \rightarrow1$ as $n\rightarrow\infty$ by utilizing a suitable rotation.

Then by Lemma~\ref{le3} and induction method, we are able to choose a decreasing  positive sequence $(\varepsilon_n)$ such that $E_{\varepsilon_1}=\partial\mathbb{D}$ and $\lim_{n\rightarrow\infty}\varepsilon_n=0$, and a subsequence $(b_n) \subset (a_n)$ such that the following three conditions hold:
\begin{eqnarray*}
&\textup{(1)}& \left(\int_{E_{\varepsilon_n}} |(W_{\phi,\varphi} f_{b_k})'|^p dm \right)^{1/p} < 4^{-n} \delta h, \quad k = 1,\ldots, n - 1; \\
&\textup{(2)}& \left(\int_{\partial\mathbb{D} \setminus E_{\varepsilon_n}} |(W_{\phi,\varphi} f_{b_n})'|^p dm \right)^{1/p} < 4^{-n} \delta h; \\
&\textup{(3)}& \left(\int_{E_{\varepsilon_n}} |(W_{\phi,\varphi} f_{b_n})'|^p dm \right)^{1/p} > \frac{h}{2}
\end{eqnarray*}
for every $n \in \mathbb{N},$ where $\delta > 0$ is a small constant whose value will be determined later.

Now we are ready to prove that there is a $C>0$ such that the inequality $\|\sum_{j=1}^\infty c_jW_{\phi,\varphi}(f_{b_j})\|_{S^p} \geq C \|(c_j)\|_{\ell^p}$ holds. By the triangle inequality in $L^p$, we have

\begin{equation}\begin{split}\nonumber
\|\sum_{j=1}^\infty c_jW_{\phi,\varphi}(f_{b_j})\|^p_{S^p}&\geq \|\sum_{j=1}^\infty \left(c_jW_{\phi,\varphi}(f_{b_j})\right)'\|^p_{H^p}\\
&=\sum_{n = 1}^\infty \int_{E_{\varepsilon_n} \setminus E_{\varepsilon_{n+1}}}\left|\sum_{j=1}^\infty \left(c_jW_{\phi,\varphi}(f_{b_j})\right)'\right|^p dm\\
&\geq\sum_{n = 1}^\infty \Bigg(|c_n| \left(\int_{E_{\varepsilon_n} \setminus E_{\varepsilon_{n+1}}}| \left(W_{\phi,\varphi}(f_{b_n})\right)'|^p dm\right)^{\frac{1}{p}}\\
&\phantom{=}-\sum_{j \neq n}|c_j|\left(\int_{E_{\varepsilon_n} \setminus E_{\varepsilon_{n+1}}}|\left(W_{\phi,\varphi}(f_{b_j})\right)'|^p dm\right)^{\frac{1}{p}} \Bigg)^p\,.
\end{split}\end{equation}

Observe that for every $n \in \mathbb{N},$ we have
\begin{equation}\begin{split}\nonumber
&\phantom{=}\left(\int_{E_{\varepsilon_n} \setminus E_{\varepsilon_{n+1}}}| \left(W_{\phi,\varphi}(f_{b_n})\right)'|^p dm\right)^{\frac{1}{p}}\\
&=\left(\int_{E_{\varepsilon_{n}}}| \left(W_{\phi,\varphi}(f_{b_n})\right)'|^p dm-\int_{E_{\varepsilon_{n+1}}}| \left(W_{\phi,\varphi}(f_{b_n})\right)'|^p dm\right)^{1/p}\\
&\geq \left(\left(\frac{h}{2}\right)^p-\left(4^{-n-1} \delta h\right)^p\right)^{1/p}\\
&\geq \frac{h}{2}-4^{-n-1} \delta h
\end{split}\end{equation}
according to conditions (1) and (3) above, where the last estimate holds for $1\leq p<\infty$.

Moreover, by condition (1) and (2), it holds that
$$\left(\int_{E_{\varepsilon_n} \setminus E_{\varepsilon_{n+1}}}|\left(W_{\phi,\varphi}(f_{b_j})\right)'|^p dm\right)^{\frac{1}{p}} < 2^{-n-j}\delta h\quad\text{ for } j \ne n.$$

Consequently, we obtain that
\begin{equation}\begin{split}\nonumber
\|\sum_{j=1}^\infty c_jW_{\phi,\varphi}(f_{b_j})\|_{S^p}&\geq\left(\sum_{n = 1}^\infty \left(|c_n| \left(\frac{h}{2}-4^{-n-1} \delta h\right)
-2^{-n}\delta h\|(c_j)\|_{\ell^p} \right)^p\right)^{1/p}\\
&\geq \left(\sum_{n = 1}^\infty \left(|c_n| \left(\frac{h}{2}\right)
-2^{-n+1}\delta h\|(c_j)\|_{\ell^p} \right)^p\right)^{1/p}\\
&\geq \frac{h}{2}\|(c_j)\|_{\ell^p}-\delta h\|(c_j)\|_{\ell^p} \left(\sum_{n = 1}^\infty2^{-(n-1)p}\right)^{1/p}\\
&\geq h\left(\frac{1}{2}-\delta\left(1-2^{-p}\right)^{-1/p}\right)\|(c_j)\|_{\ell^p}\geq C\|(c_j)\|_{\ell^p}\,,
\end{split}\end{equation}
where the last inequality holds when we choose $\delta$ small enough.

On the other hand, we are to prove the converse inequality:
$$\|\sum_{j=1}^\infty c_jW_{\phi,\varphi}(f_{b_j})\|_{S^p} \leq C \|(c_j)\|_{\ell^p}\,.$$

By definition,
$$\|\sum_{j=1}^\infty c_jW_{\phi,\varphi}(f_{b_j})\|_{S^p}=|\sum_{j=1}^\infty c_j\phi(0)f_{b_j}(\varphi(0))|+\|\sum_{j=1}^\infty \left(c_jW_{\phi,\varphi}(f_{b_j})\right)'\|_{H^p}\,.$$
First, we note that a straightforward variant of the above procedure also gives
$$\|\sum_{j=1}^\infty \left(c_jW_{\phi,\varphi}(f_{b_j})\right)'\|_{H^p}\leq C \|(c_j)\|_{\ell^p}\,.$$
Next, when $p=1$, since $\lim_{j\rightarrow\infty}f_{b_j}(\varphi(0))=0$, it is trivial that
$$|\sum_{j=1}^\infty c_j\phi(0)f_{b_j}(\varphi(0))|\leq C \|(c_j)\|_{\ell^1}\,.$$
When $1<p<\infty$, we can choose a subsequence of $(b_j)$ (still denoted by $(b_j)$) such that $\{(1-|b_j|^2)^{1/p}\}_{j=1}^\infty\in\ell^q$, where $1/p+1/q=1$\,. Then by H\"{o}lder's inequality,
$$|\sum_{j=1}^\infty c_j\phi(0)f_{b_j}(\varphi(0))|\leq C \|(c_j)\|_{\ell^p}\,.$$
Accordingly, the desired inequality follows.

By choosing $\phi=1$ and $\varphi=z$, we obtain that
$$C \|(c_j)\|_{\ell^p}\leq \|\sum_{j=1}^\infty c_jf_{b_j}\|_{S^p} \leq C \|(c_j)\|_{\ell^p}\,.$$

Thus, we have
$$\|\sum_{j=1}^\infty c_jW_{\phi,\varphi}(f_{b_j})\|_{S^p} \geq C\|\sum_{j=1}^\infty c_jf_{b_j}\|_{S^p}\,$$
The proof is complete.
\end{proof}

\section{Proof of Theorem~\ref{th2}}
In this section, we give the proof of Theorem~\ref{th2}.
\begin{proof}[Proof of Theorem~\ref{th2}]
Since $M$ is the infinite-dimensional subspace of $S^p$ and polynomials are dense in $S^p$ (see \cite{LIN}), there exists a sequence $(f_n)$ of unit vectors in $M$ such that $f_n$ converges to $0$ uniformly on compact subsets of $\mathbb{D}$. since $W_{\phi,\varphi}$ is bounded below on $M\subset S^p$, there exists $h>0$ such that
$$\|W_{\phi,\varphi}f_n\|_{S^p}>h\,,$$
for all $n\in\mathbb{N}$\,. For $k\geq1$, denote $E_k:=\{\xi\in\partial\mathbb{D}:\ |\varphi(\xi)|\geq1-1/k\}$\,. Since by assumption, $\lim_{k\to\infty}m(E_k)=m(E_\varphi)=0$, it holds that
$$\lim_{k \to \infty}\int_{E_k}|(W_{\phi,\varphi} f_{n})'|^p dm = 0$$
for every $n\in\mathbb{N}$.
Moreover, since $f_n$ converges to $0$ uniformly on compact subsets of $\mathbb{D}$, it follows that
$$\lim_{n \rightarrow \infty}\int_{\partial\mathbb{D}\setminus E_k}|(W_{\phi,\varphi} f_{n})'|^p dm = 0$$
for every $k\in \mathbb{N}.$

The remainder of the proof is an argument that goes exactly as the proof of Theorem~\ref{th1}, so we omit it. Thus, the proof is complete.
\end{proof}

\section{Proof of Theorem~\ref{th3}}
In this last section, we give a proof for Theorem~\ref{th3}.
\begin{proof}[Proof of Theorem~\ref{th3}]
We define the subspace $S^p_0$ of $S^p$ by
$$S^p_0:=\{f\in S^p:\ f(0)=0\}\,.$$
Then the integral operator $T_z:\ f\mapsto\int_0^zf(\zeta)d\zeta$ is an isometric isomorphism from $H^p$ onto $S^p_0$\,. Then the weighted composition operator $W_{\phi,\varphi}$ on $S^p_0$ is unitary similar to the operator 
$$T:=W_{\phi',\varphi}\circ T_z+W_{\phi\varphi',\varphi} \text{ on } H^p\,.$$

By Lemma~\ref{le2}, \cite[Theorem~2.1]{CH} and the expression of the operator $T$, we see that $W_{\phi,\varphi}$ is bounded on $S^p$ if and only if $\phi\in S^p$ and $W_{\phi,\varphi}$ is bounded on $S^p_0$\,.

Now, for the operator $W_{\phi\varphi',\varphi}$ on $H^p$, we can deduce from the proof of \cite[Theorem~2]{LMN} that there exists a sequence of integers $(n_k)$ satisfying $\inf_k(n_{k+1}/n_{k})>1$ and a positive constant $C$ such that
$$\|\sum_{k}c_kW_{\phi\varphi',\varphi}(e_{n_k})\|_{H^p}\geq C\|(c_k)\|_{\ell^2}\,,$$
where $e_{n_k}:=z^{n_k}$ is the unit vector in $H^p$. Since the operator $W_{\phi',\varphi}\circ T_z: H^p\to H^p$ is compact (it is equivalent to the compactness of $W_{\phi',\varphi}:S^p\to H^p$, which is claimed by Lemma~\ref{le2}), then for any $\varepsilon>0$, there exists a subsequence of $(n_k)$ (still denoted as $(n_k)$) such that
$$\|\sum_{k}c_kW_{\phi',\varphi}\circ T_z(e_{n_k})\|_{H^p}\leq \varepsilon\|(c_k)\|_{\ell^2}\,.$$
Thus,
$$\|\sum_{k}c_kT(e_{n_k})\|_{H^p}\geq C\|(c_k)\|_{\ell^2}\,,$$
which implies that the weighted composition operator $W_{\phi,\varphi}$ on $S^p_0$ is bounded below:
$$\|\sum_{k}c_kW_{\phi,\varphi}(g_{n_k})\|_{S^p}\geq C\|(c_k)\|_{\ell^2}\,,$$
where $g_{n_k}:=T_z(e_{n_k})$ is the unit vector in $S^p_0$\,.

Since Paley's theorem (see \cite{DUREN}) implies that the closed linear span $M:=\{e_{n_k}: k\geq1\}$ in $H^p$ is isomorphic to $\ell^2$, which implies that the closed linear span $T_z(M)=\{g_{n_k}: k\geq1\}$ in $S^p_0$ is isomorphic to $\ell^2$\,. Hence, $W_{\phi,\varphi}$ fixes an isomorphic copy of $\ell^2$ in $S^p_0$\,. 

Accordingly, it follows that $W_{\phi,\varphi}$ fixes an isomorphic copy of $\ell^2$ in $S^p$ since $S^p_0\subset S^p$ and $\|W_{\phi,\varphi}\|_{S^p}\geq\|W_{\phi,\varphi}\|_{S^p_0}$, which is the desired result.
\end{proof}

\section{The strict singularity of $W_{\phi,\varphi}$ on $S^\infty$}
Here we show that the claims of theorem~\ref{th1} is still true for the case of $p=\infty$\,. We have known that the weighted composition operator $W_{\phi,\varphi}$ on $S^{\infty}_0$ is unitary similar to the operator
$$T:=W_{\phi',\varphi}\circ T_z+W_{\phi\varphi',\varphi} \text{ on } H^\infty\,.$$
It follows from \cite{CHD} that any weakly compact weighted composition operator on $H^\infty$ is compact. Since by Lemma~\ref{le2} the operator $W_{\phi\varphi',\varphi}$ is compact on $H^\infty$, it holds that $T$ is weakly compact on $H^\infty$ is and only if $T$ is compact on $H^\infty$\,. Moreover, Bourgain \cite{JB} established that a bounded linear operator on $H^\infty$ is weakly compact if and only if it does not fix any copy of $\ell^\infty$\,. Thus, $T$ is compact on $H^\infty$ if and only if it does not fix any copy of $\ell^\infty$. Therefore, the weighted composition operator $W_{\phi,\varphi}$ on $S^{\infty}_0$ is compact if and only if it does not fix any copy of $\ell^\infty$. 

By Lemma~\ref{le2}, \cite[Theorem~2.1]{CH} and the expression of the operator $T$, we see that $W_{\phi,\varphi}$ is bounded on $S^{\infty}$ if and only if $\phi\in S^{\infty}$ and $W_{\phi,\varphi}$ is bounded on $S^{\infty}_0$\,. Moreover, under the assumption for the boundedness of $W_{\phi,\varphi}$ on $S^{\infty}$, $W_{\phi,\varphi}$ is bounded on $S^{\infty}$ if and only if $W_{\phi,\varphi}$ is bounded on $S^{\infty}_0$\,.

Therefore, the weighted composition operator $W_{\phi,\varphi}$ on $S^{\infty}$ is compact if and only if it does not fix any copy of $\ell^\infty$.
In particular, the noncompact operator $W_{\phi,\varphi}$ on $S^{\infty}$ is not strictly singular, that is, strict singularity of bounded operator $W_{\phi,\varphi}$ on $S^{\infty}$ coincides with its compactness.

\end{document}